\documentclass[12pt]{amsart}
\usepackage{amssymb,amscd, hyperref, epsf}
\usepackage{graphicx, xcolor}  \usepackage[normalem]{ulem}

\begin{document}

\def\Z{\mathbb Z}
\def\R{\mathbb R}
\def\C{\mathbb C}
\def\Q{\mathbb Q}
\def\P{\mathbb P}
\def\N{\mathbb N}
\def\L{\mathbb L}
\def\HH{\mathrm{H}}
\def\ss{X}

\def\conj{\overline}
\def\Beta{\mathrm{B}}
\def\const{\mathrm{const}}
\def\ov{\overline}
\def\wt{\widetilde}
\def\wh{\widehat}
\def\red{\color{red}}

\renewcommand{\cr}{\mathop{\mathrm{cr}}\nolimits}
\renewcommand{\Im}{\mathop{\mathrm{Im}}\nolimits}
\renewcommand{\Re}{\mathop{\mathrm{Re}}\nolimits}
\newcommand{\codim}{\mathop{\mathrm{codim}}\nolimits}
\newcommand{\id}{\mathop{\mathrm{id}}\nolimits}
\newcommand{\Aut}{\mathop{\mathrm{Aut}}\nolimits}
\newcommand{\lk}{\mathop{\mathrm{lk}}\nolimits}
\newcommand{\sign}{\mathop{\mathrm{sign}}\nolimits}
\newcommand{\pt}{\mathop{\mathrm{pt}}\nolimits}
\newcommand{\rk}{\mathop{\mathrm{rk}}\nolimits}
\newcommand{\SKY}{\mathop{\mathrm{SKY}}\nolimits}
\newcommand{\st}{\mathop{\mathrm{st}}\nolimits}
\def\Jet{{\mathcal J}}
\def\FC{{\mathrm{FCrit}}}
\def\sS{{\mathcal S}}
\def\lcan{\lambda_{\mathrm{can}}}
\def\ocan{\omega_{\mathrm{can}}}
\newcommand{\proj}{\mathop{\mathrm{proj}}\nolimits}
\newcommand{\cg}{\mathop{\mathrm{cg}}\nolimits}
\newcommand{\vcg}{\mathop{\mathrm{vcg}}\nolimits}

\def\ds{\displaystyle}

\newtheorem{mainthm}{Theorem}
\newtheorem{thm}{Theorem}[section]
\newtheorem{lem}[thm]{Lemma}
\newtheorem{prop}[thm]{Proposition}
\newtheorem{cor}[thm]{Corollary}
\newtheorem{conjecture}{Conjecture}

\theoremstyle{definition}
\newtheorem{exm}[thm]{Example}
\newtheorem{rem}[thm]{Remark}
\newtheorem{df}[thm]{Definition}

\title{Manturov Projection for Virtual Legendrian Knots in $ST^*F$}
\author{Vladimir Chernov, Rustam Sadykov}
\date{\today}
\address{Department of Mathematics, 6188 Kemeny Hall,
Dartmouth College, Hanover, NH 03755, USA}
\email{Vladimir.Chernov@dartmouth.edu}
\address{Department of Mathematics, 
138 Cardwell Hall,
1228 N Martin Luther King Jr Drive, Kansas State University, Manhattan, KS 66506}
\email{sadykov@ksu.edu}
\begin{abstract}
Kauffman virtual knots are knots in thickened surfaces $F\times \R$ considered up to isotopy, stabilizations and destabilizations, and diffeomorphisms of $F\times \R$ induced by orientation preserving diffeomorphisms of $F$. Similarly,  virtual Legendrian knots,  introduced by Cahn and Levi~\cite{CahnLevi}, are Legendrian knots in $ST^*F$ with the natural contact structure. Virtual Legendrian knots are considered up to isotopy, stabilization and destabilization of the surface away from the front projection of the Legendrian knot, as well as up to contact isomorphisms of $ST^*F$ induced by  orientation preserving diffeomorphisms of $F$. 

We show that there is a projection operation $\proj$ from the set of virtual isotopy classes of Legendrian knots to the set of isotopy classes of Legendrian knots in $ST^*S^2$. This projection is obtained by substituting some of the classical crossings of the front diagram with virtual crossings. It restricts to the identity map on the set of virtual isotopy classes of classical Legendrian knots. In particular, the projection $\proj$ extends invariants of Legendrian knots to invariants of virtual Legendrian knots. Using the projection $\proj$, we show that the virtual crossing number of every classical Legendrian knot equals its crossing number. We also prove that the virtual canonical genus of a Legendrian knot is equal to the canonical genus.

The construction of $\proj$ is inspired by the work of Manturov. 
\end{abstract}
\maketitle

\section{Introduction}
A {\it Gauss Diagram} of a knot $K$ in $\R^3$ is a circle parametrizing the knot. For each double point in the knot diagram, the two preimages on the circle are connected by a chord oriented from the preimage of the upper strand to the preimage of the lower strand. These chords are equipped with the sign $\pm$ corresponding to the sign of the double point in the diagram. The Reidemeister moves for knot diagrams can be expressed in terms of Gauss Diagram moves, as described by Polyak and Viro in \cite{PolyakViro}. The equivalence classes of Gauss Diagrams modulo these moves are {\em virtual knots,\/} see Kauffman~\cite{Kauffman}.

As proved by Goussarov, Polyak, and Viro~\cite{GPV} (see also Kauffman~\cite{KauffmanTalks}), two non isotopic knots  do not become equivalent in the virtual category. Carter, Kamada, Saito~\cite{CKS} proved that virtual knots can be viewed as equivalence classes of knots in a thickened surface $F\times \R$ modulo isotopies and stabilizations/destabilizations of $F$ (adding and deleting handles) away from the knot projection, as well as automorphisms of  $F\times \R$ of the form $\alpha\times \id_\R$ where $\alpha$ is an orientation preserving diffeomorphism of $F$.

 We say that a knot $K$ in a thickened surface $F\times K$ is \emph{irreducible} if  there is no destabilization of $F$ away from the projection of $K$ to $F$. 
Given a knot $K$ in a thickened surface $F\times \R$, the Kuperberg theorem~\cite{Kuperberg} asserts that $K$ is virtually isotopic to an irreducible knot $K'$ in a thickened surface $F'\times \R$. Furthermore, the pair $(F'\times \R, K')$ is unique up to an automorphism of $F'\times \R$ induced by orientation preserving automorphism of $F'$, and the surface $F'$ can be obtained from $F$ using only destabilizations and orientation preserving automorphisms.

A {\it contact form\/} on an odd dimensional manifold $Z^{2k+1}$ is a $1$-form $\alpha$ such that $\alpha \wedge d\alpha^{k}$ is nowhere zero. A {\em cooriented contact structure} on $Z^{2k+1}$ is a cooriented (transversally oriented) hyperplane distribution such that it is the kernel of some contact form with the coorientation pointing into the half spaces where $\alpha$ is positive.

A contact structure is characterized as maximally non-integrable hyperplane distribution, meaning that the maximum dimension of an everywhere tangent submanifold is $k$ in $(Z^{2k+1}, \ker \alpha)$. Such a submanifold $L^k$ is called {\it Legendrian.\/}

The {\it spherical cotangent bundle} $ST^*M$ of $M$ is the space of all nonzero linear functionals $T_pM\to \R$ considered up to multiplication by a positive scalar. It is equipped with the natural projection $\pi:ST^*M\to M$. The space $ST^*M$ has a {\it natural contact structure\/} where the contact plane at a linear functional $l\in ST^*M$ is $d\pi^{-1} (\ker l)$.

A Legendrian submanifold $L\subset ST^*M$ is fully described by its naturally cooriented  {\it front} projection $\pi(L)\subset M$. Note that in general the front of $L$ is not immersed and has various Legendrian singularities such as cusps.

{\it Virtual Legendrian knots\/} were introduced by Cahn and Levi~\cite{CahnLevi}. These are Legendrian knots $L$ in $ST^*F$, where $F$ is a surface (or more generally in $ST^*M$, where $M$ is a higher dimensional manifold), considered modulo Legendrian isotopies and modifications of $ST^*F$ induced by stabilization and destabilization of $F$ away from the front projection of $L$, as well as contact automorphisms of $ST^*F$ induced by orientation preserving diffeomorphisms of $F$.  They showed that all the alternative view points of the classical virtual knot theory in $F\times \R$ also hold for virtual Legendrian knots in $ST^*F$. That is they can be viewed as equivalence classes of formal generalized Gauss diagrams with special symbols for various cusps and crossing points of the front projection and they can be viewed as cooriented front diagrams with extra virtual crossings.

The authors of this paper proved a generalization of the Kuperberg theorem to virtual Legendrian knots~\cite{ChernovSadykovVirtual} and, in particular, concluded that two non-isotopic Legendrian knots in $ST^*S^2$ are also not virtually Legendrian isotopic. 

We also applied~\cite{ChernovSadykovConjectures} the virtual Legendrian knot theory to the study of causality in Borde-Sorkin spacetimes~\cite{Borde, Sorkin} and made a few related conjectures in the spirit of the works of the first author and Nemrovski~\cite{Chernov, CN1, CN2, CN3} that were motivated by the Low Conjecture~\cite{Low0, Low1, Low2, Low3} and the Legendrian Low Conjecture of Natario and Tod~\cite{NatarioTod} from the Penrose school of General Relativity~\cite{Penrose}.

\section{Crossing number and canonical genus corollaries of the Projection Theorem}

 In section~\ref{s:proj} we will prove the Projection Theorem (Theorem~\ref{mainprojection}) for virtual Legendrian knots asserting that there is a map $upr$ that associates a classical Legendrian knot diagram in $\R^2$ to each virtual Legendrian knot diagram in $\R^2$. Furthermore, the generalized Gauss diagram of the Legendrian knot $upr(K)$ is obtained from the generalized Gauss diagram of the virtual Legendrian knot $K$ by erasing some of the chords and leaving the other chords intact,  and  if $K$ is virtual Legendrian isotopic to a classical Legendrian knot, then $upr(K)$ belongs to the virtual Legendrian isotopy class of $K$.  

In this section we will use the Projection Theorem to deduce two consequences, namely, that the virtual crossing number of every classical Legendrian knot equals its crossing number, and that the virtual canonical genus of a Legendrian knot coincides with the canonical genus.

\begin{df} Given a knot diagram $D$ of the knot $K$ its {\it canonical genus} $g(D)$ is the genus of a  \emph{canonical surface $S$} obtained as follows: smoothen the classical crossings of $D$ using orientation of the knot branches and glue the resulting ovals with disks. Then glue the bands to the disks at each crossing.  There is an obvious copy $K'$ of the knot $K$ over the surface $S$. We say that $(K', S)$ is a \emph{canonical representative} of the knot diagram $D$. 
The {\it canonical genus\/} $\cg(K)$ is the minimum of $g(D)$ over all the knot diagrams of $D.$  Clearly this procedure gives also the {\it virtual canonical genus $\vcg(K_v)$} of a virtual knot $K_v$ with the minimum taken over the class of virtual knot diagrams of $K_v.$ It is clear that when the virtual knot $K_v$ is isotopic to a classical knot $K$ we have $\vcg(K_v)\leq \cg(K)$, see Stoimenov, Tchernov, Vdovina~\cite{SCV}. In~\cite{ChernovManturov} the first author used the Manturov Projection Theorem to prove that for a classical knot $K$ we have $\vcg(K)=\cg(K)$.
\end{df}

It is easy to see that the quantities $g(D)$ and hence $\vcg(L_v)$ and $\cg(L)$ are well defined for the classical and virtual Legendrian front diagram and we clearly have the inequality $\vcg(L_v)\leq \cg(L).$ Note that the knot $L$ does not have to be zero homologous and hence the actual Seifert surface given by the above procedure does not have to exist. Repeating the argument of~\cite{ChernovManturov} but using Projection Theorem~\ref{mainprojection} instead of the Manturov Projection Theorem, we get the following Theorem.

\begin{thm} If a virtual Legendrian knot $L_v$ is virtually isotopic to the classical Legendrian knot $L\subset ST^*S^2$ then $\vcg(L_v)=\cg(L).$
\end{thm}

Define the {\it crossing number\/} $\cr(L)$ of a Legendrian knot in $ST^*F$ to be the minimal number of crossing points in its front diagram. Similarly define the {\it virtual crossing number\/} $\cr_v(L_v)$ of a virtual Legendrian knot $L_v$ to be the minimal number of classical crossing points in its virtual front diagram.

Similar to Belousov~\cite{Belousov} definition for classical knots, define the {\it $k$-arc crossing number} $\cr^k(L)$ to be the minimal number of classical crossing numbers among all presentation of $L$ by all diagrams that can be decomposed into $k$ arcs without classical self-intersection points. Similarly, define the {\it virtual $k$-arc crossing number\/} $\cr_v^k(L)$ to be the minimal number of classical crossings among all diagrams of a virtual Legendrian knot $L$ that can be decomposed into $k$ arcs without self-intersection points.

Clearly for a Legendrian knot $L$ we have $\cr_v(L)\leq \cr(L)$ and for all positive $k\in \N$ we have $\cr_v^k(L)\leq \cr_v(L)$, and $\cr^{k+1}_v(L)\leq \cr^k_v(L)$ for all positive $k\in \N.$

\begin{thm}\label{kcrossing}
For a classical Legendrian knot $L\subset ST^*S^2$ we have $\cr_v(L)=\cr(L).$ Similarly for a classical Legendrian knot $L$ we have $\cr^{k+1}(L)\leq \cr^k(L)$ and for a Legendrian knot $L$ and positive $k\in\N$ we have $\cr^k_v(L)=\cr^k(L).$
\end{thm}

\begin{proof}
Take a virtual diagram of $L$ with the minimal number of classical crossings. By the projection Theorem~\ref{mainprojection} we can erase some of these crossings and get the same classical Legendrian knot, thus we have the other inequality $\cr_v(L)\geq \cr(L).$

The remaining inequality for the $k$-arc virtual and $k$-arc classical crossing number is proved similarly.
\end{proof}    

\begin{rem}
For classical knots in $\R^3$ the $k$-arc crossing number was introduced by Belousov~\cite{Belousov}. The statements similar to Theorem~\ref{kcrossing} for Kauffman rather than Legendrian virtual knots was first proved by Belousov, Chernov, Malyutin and Sadykov~\cite{BCMS}.    
\end{rem}

\section{Homological parity}

Let $D$ be the generalized Gauss diagram of a Legendrian knot $K$ in the spherical cotangent bundle $ST^*F$ of a closed orientable surface $F$. Let $\pi\colon ST^*F\to F$ denote the projection. 
We will denote the set of double points of the projected curve $\pi(K)$ by $V(K)$ or $V(D)$. The set $V(K)$ coincides with the set of chords of the generalized Gauss diagram $D$, and with the set of classical crossings of the planar diagram of the virtual knot.

 A \emph{parity} $p$ on the isotopy class $\mathcal{K}$ of Legendrian knots in $ST^*F$ with coefficients in an abelian group $A$ is a family of functions $p_K\colon V(K)\to A$ parametrized by the Legendrian knots  $K$ in the isotopy class $\mathcal{K}$ such that  
\begin{itemize}
\item if the Legendrian knot $K'$ is obtained from a Legendrian knot $K$ by a Legendrian Reidemeister move, then $p_K(v)=p_{K'}(v)$ for all crossings $v$ that are not involved in the Reidemeister move, 
\item $p_K(v_1)+p_K(v_2)=0$ if there is a second Legendrian Reidemeister move that cancels the two crossings $v_1$ and $v_2$, and
\item $p_K(v_1)+p_K(v_2)+p_K(v_3)=0$ if there is a third Legendrian Reidemeister move involving the crossings $v_1$, $v_2$ and $v_3$. 
\end{itemize}

We will next review the definition of the \emph{homological parity} given by Ilyutko, Manturov, and Nikonov in \cite{IMN}. The homological parity $p_u$ is a function on the set $V(D)$ with values in $H_1(F; \Z_2)/[K]$ where $[K]$ stands for the homology class represented by $\pi(K)$. Given a double point $v$ of the projected curve $\pi(K)$, the orientation of $K$ defines a unique smoothing of $\pi(K)$ at $v$ that breaks the curve $\pi(K)$ into two components $K_{v,1}$ and $K_{v,2}$. More precisely, suppose that the curve $\pi(K)$ is parametrized by a map $\varphi\colon S^1\to F$ such that $\varphi(p_1)=\varphi(p_2)=v$. Let $\alpha$ and $\beta$ be the two arcs in $S^1$ between $p_1$ and $p_2$. Then the two components $K_{v,1}$ and $K_{v,2}$ of $K$ are given by $\varphi(\alpha)$ and $\varphi(\beta)$.

The equivalence class of the homology class $[K_{v,1}]$ coincides with that of $[K_{v,2}]$ in $H_1(F; \Z_2)/[K]$ and will be denoted by $[v]$. The homological parity associates with each crossing $v$ the equivalence class $[v]$ in $H_1(F; \Z_2)/[K]$. By \cite[Lemma 3.3]{IMN}, the homological parity is a well-defined parity that reduces to a parity on homotopy classes of generically immersed curves in $F$. 

The $\Z_2$-valued homology parity $p_u^{\Z_2}$ is obtained from the homology parity $p_u$ by postcomposing the functions $p_u$ with the function $H_1(F; \Z_2)/[K]\to \Z_2$ that takes the zero element to the zero element and non-trivial elements to non-trivial elements. We note that this function is not a homomorphism. We will use the same notation $p_u^{\Z_2}$ for parities over all closed oriented surfaces $F$ and all isotopy classes of knots in $ST^*F$. Then, strictly speaking, $p_u^{\Z_2}$ represents a family of parities parametrized by closed oriented surfaces $F$ and isotopy classes of knots in $ST^*F$. Slightly abusing notation, we will simply refer to $p_u^{\Z_2}$ as a parity.

The parity $p_u^{\Z_2}$ gives rise to the definition of odd and even crossings. Namely, 
we say that the crossing $v$ is \emph{even} if $[v]=0$, and \emph{odd} otherwise. 
Recall that a stabilization of a knot $K$ in $ST^*F$ is obtained by stabilizing the surface $F$ away from the front projection $\pi(K)$. We say that the stabilization of $F$ is \emph{trivial}, if it is obtained by stabilizing $F$ by removing two open discs in the interior of a disc in $F\setminus \pi(K)$, and then attaching a handle along the two created boundary components. Equivalently, a trivial stabilization of a surface $F$ is obtained by removing an open disc in $F\setminus \pi(K)$ and attaching along the new boundary component a copy of a torus without an open disc. 
The operation inverse to a trivial stabilization is a \emph{trivial destabilization}. 

\begin{lem}\label{l:pr5.1} The parity $p_u^{\Z_2}$ does not change under trivial stabilizations and trivial destabilizations of Legendrian knots.  
\end{lem}
\begin{proof}
    Suppose that a Legendrian knot $K'$ in $F'$ is obtained by a trivial stabilization of the Legendrian knot $K$ in $F$. In particular, the surface $F'$ is obtained from $F$ by removing two small open discs in the interior of 
a disc in the complement $F\setminus \pi(K)$ and attaching a handle $H$ diffeomorphic to a cylinder along the created boundary components. Then $F'/H$ is homeomorphic to $F\vee S^1$, and there is a homology exact sequence of the pair $(F', H)$, 
    \[
     0\longrightarrow H_1(H; \Z_2)\longrightarrow H_1(F'; \Z_2)\longrightarrow H_1(F; \Z_2)\oplus \Z_2\longrightarrow 0. 
    \]
    Since $H_1(H; \Z_2)=\Z_2$ is generated by the meridian of the handle, and the front projection $\pi(K)$ is disjoint from $H$, we conclude that there is a short exact sequence of vector spaces
    \[
     0\longrightarrow H_1(H; \Z_2)\longrightarrow H_1(F'; \Z_2)/[K']\longrightarrow H_1(F; \Z_2)/[K]\oplus \Z_2\longrightarrow 0.
    \]

    Every crossing $v$ in $K$ corresponds to a crossing $v'$ in $K'$, which, in its turn, determines two halves $K'_{v',1}$ and $K'_{v',2}$. In view of the above homology exact sequence, if $[K'_{v',1}]=[K'_{v',2}]$ is trivial in $H_1(F'; \Z_2)/[K']$, then it is also trivial in $H_1(F; \Z_2)/[K]$. On the other hand, if it is non-trivial in $H_1(F'; \Z_2)$, but it is in the kernel of the projection 
    \[
    H_1(F'; \Z_2)/[K']\longrightarrow H_1(F; \Z_2)/[K],
    \]
    then either $[K']$ has a non-trivial component in $H_1(S^1; \Z_2)=\Z_2$ or it
    lifts to a non-trivial class in $H_1(H; \Z_2)$.
    Since the knot is disjoint from the attached handle the former is impossible. The latter is impossible since 
    the knot is disjoint from the attached handle, and since the stabilization is trivial.  Thus, $[K_{v, 1}]$ is trivial if and only if $[K'_{v', 1}]$ is trivial.   
\end{proof}

\begin{rem} 
In general a stabilization/destabilization may change the parity of a crossing. For example, let $F'$ be a sphere with two handles. Its homology group $H_1(F'; \Z_2)$ is isomorphic to $\Z_2[m_1, \ell_1, m_2, \ell_2]$, where $m_1$ and $m_2$ are the homology classes of the meridians of the two handles, while $\ell_1$, $\ell_2$ are the homology classes of the two longitudes. There is a knot $K'$ in $ST^*F'$ whose projection $\pi(K')$ is a curve with a unique crossing $v$ that breaks $\pi(K')$ into two loops: a loop $K'_{v,1}$ representing the homology class $m_1$, and a loop $K'_{v,2}$ representing the homology class $m_2$. Then in the quotient group $G'=H_1(F'; \Z_2)/[K']$ the classes $m_1$ and $m_2$ represent the same equivalence class, the group $G'$ itself 
is isomorphic to $\Z_2[m_1, \ell_1, \ell_2]$, and the classes $[K_{v, 1}]=[K_{v,2}]=m_1$ are non-trivial. In particular, the crossing $v$ is odd. 

We may choose $K'$ so that its projection is disjoint from a destabilizing curve $\alpha$ on $F'$ representing the homology class $m_1$. Let $F$ denote the surface obtained from $F'$ by the destabilization along the curve $\alpha$. Then the group $G=H_1(F; \Z_2)/[K]$ is isomorphic to $\Z_2[m_2, \ell_2]/[K]$, where $[K]$ is the homology class of the projection $\pi(K)$ of the destabilized knot. The destabilized knot $K$ still has a unique crossing which we continue denoting $v$. The corresponding two loops $K_{v, 1}$ and $K_{v,2}$ now represent the homology classes $0$ and $m_2$ respectively. Therefore, the group $G$ is isomorphic to $\Z_2[\ell_2]$, and the loops $K_{v, 1}$ and $K_{v,2}$ represent the zero class in $G$. Consequently, the crossing $v$ of $K$ is even.  

Thus, indeed, in this case the destabilization of $K'$ turns the odd crossing $v$ into an even crossing. 
\end{rem}

There is an equivalent definition of the homological parity given in \cite{ManturovProjection}. Namely, 
a regular neighborhood of the projection $\pi(K)$ of $K$ in $F$ is a surface $S'$ with boundary. It consists of crosses (corresponding to classical crossings in $D$) and bands along $\pi(K)$. The boundary components of $S'$ are called \emph{pasted cycles}. The projection of a regular neighborhood of $\pi(K)$ to $\pi(K)$ takes every pasted cycle $\ell$ to a curve $\ell'$ in $\pi(K)$. By definition, the group $G(D)$ is the abelian group generated by elements $a_i$ corresponding to the crosses in $S'$ with two types of relations. First, we have $2a_i=0$ for every generator $a_i$, and, second, for every pasted cycle $\ell$, the sum of elements corresponding to crossings along the projected pasted cycle $\ell'$ is trivial. In particular, every crossing corresponds to an element in the group $G(D)$. The parity $p$ is defined by associating with each crossing $a_i$ the corresponding element $a_i$ in the group $G(D)$. 

 We say that a crossing is {\it even,\/} if it corresponds to a trivial element in the group $G(D)$. Otherwise, the crossing is said to be {\it odd.\/} 

\begin{thm}[Manturov] 
The two definitions of odd and even crossings are equivalent under the assumption that the complement $F\setminus \pi(K)$ is a disjoint union of discs. This assumption is clearly satisfied for any irreducible Legendrian knot $K$,  and in the case where $(F, K)$ is the canonical representative of a virtual knot diagram.
\end{thm}
\begin{proof} There is a homomorphism $h\colon G(D)\to H_1(F; \Z_2)/[K]$ defined by taking the element of a crossing $v$ in $G$ to the equivalence class $[v]$ representing $[K_{v, 1}]=[K_{v, 2}]$. The  homomorphism $h$ is well-defined since every element in the target group is of order $2$, and for every pasted cycle $\gamma$, we have $\sum_{v_i\in \gamma} [v_i]=0$ for the crossings $v_i$ along $\gamma$. To prove the latter, we observe that $\gamma$ inherits an orientation from $K$, and can be parametrized by length to produce a piecewise smoothly immersed curve with singularities at crossings where the velocity vector of $\gamma$ rotates. There is an immersed parametrized curve obtained from $\gamma$ by attaching $\pm K_{v_i,1}$ or $\pm K_{v_i,2}$ for each $v_i$. It has no rotation vertices and therefore it is a multiple of $K$, see \cite[Lemma 4.6]{IMN}. In other words, the equivalence class of $[\gamma]+\sum_{v_i\in \gamma} [v_i]$ is trivial in $H_1(F; \Z_2)/[K]$, where the sum ranges over the crossings along $\gamma$. Finally, we observe that the homology class $[\gamma]$ is trivial since each pasted cycle belongs to the complement $F\setminus \pi(K)$ which is a disjoint union of discs.
Thus, indeed, the homomorphism $h$ is well-defined.

The group $H_1(\pi(K); \Z_2)/[K]$ is isomorphic to the group $\Z_2\langle [v_i]\rangle$, while the group $H_1(F; \Z_2)/[K]$  is obtained from 
\[
H_1(\pi(K); \Z_2)/[K]=\Z_2\langle [v_i]\rangle
\]
by factoring out 
the relations corresponding to the $2$-cells in the minimal relative CW-structure of $(F, \pi(K))$. 
More precisely, the group $H_1(F; \Z_2)/[K]$ is obtained from 
$H_1(\pi(K); \Z_2)/[K]$
by factoring by the subgroup generated by all elements $\gamma$ that are boundaries of the $2$-cells in the relative CW-structure. We already know that the elements $\gamma+\sum_{v_i\in \gamma} [v_i]$ are trivial in $H_1(F; \Z_2)/[K]$, and therefore, $\langle \gamma\rangle$ is generated by the elements 
\[
(\gamma+\sum_{v_i\in \gamma} [v_i]) -\gamma= \sum_{v_i\in \gamma} [v_i]. 
\]
Consequently, the group $H_1(F; \Z_2)/[K]$ is isomorphic to $G$. 
\end{proof}

The homological parity $p_u$ has a universal property that for any other parity $p$ for Legendrian knots in $ST^*F$ with coefficients in an abelian group $A$ there exists a unique homomorphism $\rho\colon H_1(F; \Z_2)/[K]\to A$ such that $p=\rho\circ p_u$. In particular, if a crossing $v$ is even with respect to the homological parity, then it is even with respect to any other parity. This is proved for generically immersed curves, \cite[Theorem 4.4]{IMN}, but the proof extends to knots in a thickened surface $F\times [I]$, see a comment before \cite[Lemma 4.8]{IMN}, and to Legendrian knots in $ST^*F$.

\section{The Projection Theorem for virtual Legendrian knots}\label{s:proj}

 Given a  diagram of a virtual knot $K$ in $\R^2$ with virtual crossings marked by small circles, 
the projection $upr(K)$ is defined using the homological parity $p_u$
by iteratively
\begin{itemize}
\item constructing the canonical surface $F$ for $K$, and
\item replacing odd classical crossings in the diagram of $K$ with virtual crossings.
\end{itemize}

\begin{thm}\label{mainprojection} The projection $upr$ takes a diagram of a virtual knot $K$ in $\R^2$ to a diagram of a classical knot in $\R^2$ so that: 
\begin{itemize}
\item Every classical crossing of $upr(K)$ corresponds to a classical crossing of $K$. In other words, the generalized Gauss diagram of $upr(K)$ is obtained from one of $K$ by erasing some of the chords.
\item If two diagrams represent the same virtual Legendrian isotopy class of Legendrian knots, then their images under the projection $upr$ are diagrams representing the same Legendrian isotopy class of knots. In particular, $upr$ defines a map from Legendrian virtual isotopy classes of virtual Legendrian knots to Legendrian isotopy classes of Legendrian knots. 
\item If a virtual Legendrian knot $K$ is virtual Legendrian isotopic to a classical Legendrian knot, then $upr(K)$ is virtual Legendrian
isotopic to $K$.
\end{itemize}

\end{thm}
\begin{proof} 
To prove the first statement, we note that given a knot $K$, the definition of $upr(K)$ does not involve Reidemeister moves of any intermediate knots.

Next, if $K'$ is obtained from $K$ by the first Legendrian Reidemeister move and has an extra crossing $v$, then $v$ is even, and therefore, the generalized Gauss diagrams of $K$ and $K'$ are the same. 

Suppose now that $K'$ is obtained from $K$ by the second Legendrian Reidemeister move (cusp pinching a branch or safe tangency move)  by introducing two new classical crossings $v_1$ and $v_2$. By the second axiom of parities, either both of these crossings are even, or both are odd. If both crossings are odd, then the generalized Gauss diagrams of $K$ and $K'$ coincide. If both crossings are even, then $upr(K')$ is obtained from $upr(K)$ by the second virtual Legendrian Reidemeister move. 

Finally, suppose that $K'$ is obtained from $K$ by the third Legendrian Reidemeister move. Let $v_1, v_2$ and $v_3$ denote the three classical crossings involved in the move. 
The curve $\pi(K')$ consists of a triangle with vertices $v_1, v_2$ and $v_3$, and three arcs $\alpha$, $\beta$ and $\gamma$ with end points at the three vertices $v_1, v_2, v_3$. We will denote the segment of the triangle between the two vertices $v_i$ and $v_j$ by $v_{ij}$. 

\begin{lem}\label{l:abg} Each of $\alpha$, $\beta$ and $\gamma$ is one of $K'_{v_i, j}$.
\end{lem}
\begin{proof}  We may assume that the curve $\pi(K')$ is the concatenation of arcs $\alpha* v_{i_1i_2} *\beta * v_{i_3i_4} * \gamma * v_{i_5, i_6}$, where $i_1,..., i_6$ are numbers in $\{1, 2, 3\}$, and the base point is the point $v_{i_6}$. Suppose, for example, that $i_6=1$. Then 
\[
 K'_{v_1, 1} * K'_{v_1, 2} = \alpha* v_{i_1i_2} *\beta * v_{i_3i_4} * \gamma * v_{i_5, 1}.
\]
We note that $i_1\ne i_2$, $i_3\ne i_4$ and $i_5\ne i_6$ since the curve $\pi(K')$ contains all three segments of the Reidemeister triangle. 
Since the single segment $v_{i_5, 1}$ cannot be the loop $K_{v_1,2}'$, we deduce that 
$K'_{v_1,1}$ can only be 
\begin{itemize}
\item the arc $\alpha$, 
\item the arc $\alpha*v_{i_1,i_2}$, 
\item the arc $\alpha* v_{i_1i_2} *\beta$, or 
\item the arc $\alpha* v_{i_1i_2} *\beta * v_{i_3i_4}$.
\end{itemize}
In the first two cases $K'_{v_1, 1}$ contains only one arc $\alpha$, while $K'_{v_1,2}$ contains two arcs $\beta$ and $\gamma$. In the third and forth cases we can relabel $K'_{v_1, 1}$ and $K'_{v_1,2}$ so that $K'_{v_1,1}$ again contains only one arc out of the three arcs $\alpha, \beta$ and $\gamma$. Similarly, for $i=2$ and $i=3$, by relabeling $K'_{v_i, 1}$ with $K'_{v_i, 2}$ when necessary, we may assume that each of $K'_{v_i, 1}$ contains only one of the arcs $\alpha$, $\beta$, $\gamma$, while $K'_{v_i, 2}$ contains the other two arcs from the set $\{\alpha, \beta, \gamma\}$.  We note that the loops $K'_{v_i, 1}$ and $K'_{v_j,1}$ may contain a common edge $v_{ij}$.

To prove Lemma~\ref{l:abg}, it suffices to show that if $K_{v_i, 1}$ contains one of the arcs in \{$\alpha,\beta, \gamma$\}, then non of $K_{v_j, 1}$ with $j\ne i$ contains the same arc.

Suppose that $K_{v_1,1}$ is $\alpha$, in particular, the end points of $\alpha$ are the same. Then the orientation of $\alpha$ determines the orientation of $v_{12}$ and $v_{13}$. Therefore, the values of $i_2, i_3, ...$ in 
\[
K'_{v_1, 2} = v_{1i_2} *\beta * v_{i_3i_4} * \gamma * v_{i_5, 1}
\]
are determined by the orientation of $v_{23}$. In both cases $\alpha\ne K_{v_i, 1}$ for $i\ne 1$. The same conclusion is true when the end points of $\alpha$ are distinct, i.e, $K_{v_1,1}=\alpha*v_{i_1, i_2}$, where we also consider essentially two cases. This concludes the proof of Lemma~\ref{l:abg}.
\end{proof}

Let $\rho\colon H_1(F; \Z_2)/[K']\to \Z_2$ denote the projection.
Then using $p_u^{\Z_2}=\rho\circ p_u$, in view of Lemma~\ref{l:abg}, we have 
\[
  \rho[\alpha] = \rho[\beta] + \rho[\gamma],
\]
\[
  \rho[\beta] =\rho[\alpha]+\rho[\gamma],
\]
and
\[
\rho[\gamma] = \rho[\alpha] +\rho[\beta]
\]
where $[\alpha]$, $[\beta]$ and $[\gamma]$ are elements in $H_1(F; \Z_2)/[K']$. 
We deduce that $\rho[\alpha]=\rho[\beta]=\rho[\gamma]$. In other words, either all crossings involved in the third Legendrian Reidemeister move are even, or all are odd. If all crossings are even, then $upr(K')$ is obtained from $upr(K)$ by the third Legendrian Reidemeister move. If all crossings are odd, then the diagrams of $upr(K')$ and $upr(K)$ are the same. 

Thus, the values of the projection $upr$ are the same on Legendrian knots that are Legendrian virtually isotopic.

 Let now $F$ be the canonical surface associated with a diagram of a virtual Legendrian knot $K$.
 Then $H_1(F; \Z_2)/[K]$ can be identified with the group $G(D)$ whose generators are in bijective correspondence with the crossings $a_i$. Suppose that all vertices of a Legendrian knot $K$ in $ST^*F$ are even. We recall that a crossing $a_i$ is even if the corresponding generator in $G(D)$ is trivial. Since all crossings are even, we deduce that the group $G(D)$ is trivial. Consequently, the abelian group $H_1(F; \Z_2)$ is generated by $[K]$. This is possible only if the surface $F$ is a sphere. Consequently, if all vertices of a knot are even, then the knot is classical. In particular, the values of the projection $upr$ are classical Legendrian knots in $ST^*S^2$. 

Finally, suppose that a Legendrian knot $K$ in $ST^*F$ is Legendrian virtually isotopic to a knot $K'$ in $ST^*S^2$. Then $upr(K)=upr(K')$. On the other hand, the group $H_1(F; \Z_2)/[K']$ is trivial, and therefore, all vertices of $K'$ are even. Consequently, $upr(K')=K'$.
\end{proof}

{\bf Acknowledgment:} The authors are thankful to Semen Podkorytov for many valuable comments and questions.

\end{document}